\documentclass[12pt]{article}
\usepackage{amsmath}
\usepackage{latexsym}
\usepackage{amssymb}
\newtheorem{thm}{Theorem}[section]
\newtheorem{la}[thm]{Lemma}
\newtheorem{Defn}[thm]{Definition}
\newtheorem{Exa}[thm]{Example}
\newtheorem{Remark}[thm]{Remark}
\newtheorem{prop}[thm]{Proposition}

\newtheorem{Number}[thm]{\!\!}

\newenvironment{exa}{\begin{Exa}\rm}{\end{Exa}}
\newenvironment{rem}{\begin{Remark}\rm}{\end{Remark}}
\newenvironment{numba}{\begin{Number}\rm}{\end{Number}}
\newenvironment{proof}{{\noindent\bf Proof.}}%
                  {\nopagebreak\hspace*{\fill}$\Box$\medskip\par}
\newcommand{\Punkt}{\nopagebreak\hspace*{\fill}$\Box$}
\newcommand{\wb}{\overline}

\newcommand{\mto}{\mapsto}
\newcommand{\N}{{\mathbb N}}
\newcommand{\R}{{\mathbb R}}
\newcommand{\C}{{\mathbb C}}
\newcommand{\K}{{\mathbb K}}
\newcommand{\Q}{{\mathbb Q}}
\newcommand{\bS}{{\mathbb S}}
\newcommand{\cg}{{\mathfrak g}}

\newcommand{\cL}{{\mathcal L}}
\newcommand{\cO}{{\mathcal O}}
\newcommand{\cT}{{\mathcal T}}
\newcommand{\cU}{{\mathcal U}}
\DeclareMathOperator{\Aut}{Aut}
\newcommand{\sub}{\subseteq}

\DeclareMathOperator{\GL}{GL}
\DeclareMathOperator{\SL}{SL}
\DeclareMathOperator{\id}{id}

\DeclareMathOperator{\Diff}{Diff}
\DeclareMathOperator{\Lip}{Lip}

\DeclareMathOperator{\Supp}{supp}
\DeclareMathOperator{\Gau}{Gau}
\newcommand{\pl}{{\displaystyle \lim_{\longleftarrow}\, }}
\newcommand{\dl}{{\displaystyle \lim_{\longrightarrow}\, }}
\begin{document}
\begin{center}
{\Large\bf
Completeness of infinite-dimensional\\[2mm]
Lie groups in their left uniformity}\\[7mm]
{\bf Helge Gl\"{o}ckner}\vspace{2mm}
\end{center}
\begin{abstract}
\hspace*{-6mm}We prove completeness for the main examples
of infinite-dimensional Lie groups and some related topological groups.
Consider a sequence
$G_1\sub G_2\sub\cdots$ of topological groups~$G_n$
such that~$G_n$ is a subgroup of $G_{n+1}$ and the latter induces the given topology on~$G_n$,
for each $n\in\N$.
Let $G$ be the direct limit of the sequence in the category of topological groups.
We show that $G$ induces the given topology on each~$G_n$ whenever
$\bigcup_{n\in \N}V_1V_2\cdots V_n$ is an identity neighbourhood in~$G$
for all identity neighbourhoods $V_n\sub G_n$. If, moreover,
each $G_n$ is complete, then~$G$ is complete.
We also show that the weak direct product $\bigoplus_{j\in J}G_j$
is complete for
each family $(G_j)_{j\in J}$ of complete Lie groups~$G_j$.
As a consequence, every strict direct limit $G=\bigcup_{n\in \N}G_n$ of finite-dimensional
Lie groups is complete, as well as the diffeomorphism group $\Diff_c(M)$
of a paracompact finite-dimensional smooth manifold~$M$
and the test function group $C^k_c(M,H)$, for each $k\in\N_0\cup\{\infty\}$
and complete Lie group~$H$
modelled on a complete locally convex space.
\end{abstract}
{\bf Classification:}
22E65 (primary);
22A05, 
22E67, 
46A13, 
46M40, 
58D05
\\[2.3mm]
{\Large\bf Introduction and statement of the main results}\\[4mm]
Our main goal is to study completeness for Lie groups modelled on locally convex spaces
(in the sense of~\cite{RES}, \cite{GaN}, \cite{Nee},
cf.\ also \cite{Ham}, \cite{Mic}, \cite{Mil}),
and more generally completeness of topological groups,
as far as this is useful for the main goal.
Here completeness refers to the left uniform structure on~$G$ (see, e.g.,  \cite{HaR} for the latter).
The topological groups under consideration need not be Hausdorff
(unless we say so explicitly).
It is well-known that every Lie group~$G$ modelled on a Banach space~$E$
is complete (see Proposition~1 in \cite[Chapter~III, \S1.1]{Bou}),
as the left uniform structure and the one induced by the
additive group of the Banach space coincide on some
identity neighbourhood $U\sub G$ which is homeomorphic to a closed
$0$-neighbourhood $V\sub E$.\\[2.3mm]
Projective limits of complete Hausdorff groups
being complete, this implies that many Fr\'{e}chet-Lie groups\footnote{As usual, Lie groups modelled on Fr\'{e}chet
spaces (resp., Banach spaces) are called Fr\'{e}chet-Lie groups (resp., Banach-Lie groups)
in the following.}
are complete, e.g.\ the mapping groups
\[
C^\infty(M,H)=\pl C^k(M,H)\vspace{-1mm}
\]
which are the projective limit of the Banach-Lie groups $C^k(M,H)$ for $k\in\N_0$,
for each compact smooth manifold~$M$ and Banach-Lie group~$H$.
Of course, completeness properties of locally convex spaces~$E$
(which furnish examples of abelian Lie groups $(E,+)$) are a standard topic in functional
analysis. 
Moreover, completeness properties of topological groups have been studied
(see \cite{RaD} and the references therein, also \cite{AaT}). 
However, no systematic study of completeness properties
of infinite-dimensional Lie groups is available so far.
For example, it is an open question whether completeness
of the modelling space implies completeness for a Lie group (see Problem~II.9 in the survey~\cite{Nee}).
The current article strives to develop specific tools
which enable completeness to be shown for important classes of infinite-dimensional Lie groups,
under natural hypotheses.
Our main results, Theorems~A and~B, are devoted to
completeness properties of direct limits.
Recall that an ascending sequence
\[
G_1\sub G_2\sub\cdots
\]
of topological groups $(G_n,\cO_n)$ is a \emph{direct sequence of topological groups}
if, for each $n\in\N$, the inclusion map $j_{n+1,n}\colon (G_n,\cO_n)\to (G_{n+1},\cO_{n+1})$ is a continuous
homomorphism.
If, moreover, each $j_{n+1,n}$ is a topological embedding (i.e., a homeomorphism onto its image),
then the direct sequence is called \emph{strict}.
Give $G:=\bigcup_{n\in\N}G_n$
the unique group structure for which each inclusion map $j_n\colon G_n\to G$ is a group homomorphism.
There is a finest group topology $\cO_{TG}$ on~$G$ making each~$j_n$ continuous;
$(G,\cO_{TG})$ is called the \emph{direct limit topological group}.\footnote{A topology on a group $G$ is
called a \emph{group topology} if it makes $G$ a topological group.}
The topology~$\cO_{TG}$ must not be confused with the final topology~$\cO_{DL}$ on~$G$
with respect to the inclusion maps~$j_n$, which makes $(G,\cO_{DL})$
the direct limit topological space. It is clear that $\cO_{TG}\sub\cO_{DL}$,
but examples show that equality need not hold (see, e.g., \cite{Tat}
and~\cite{Yam}).\\[2.3mm]
Let $\cO$ be a group topology on~$G$ making each~$j_n$ continuous.
Following~\cite{JFA},
we say that \emph{product sets are large} in $(G,\cO)$ if\hspace*{.3mm}\footnote{In \cite{Bk3},
$(G,\cO)$ is then said to \emph{carry the strong topology}.}
\begin{equation}\label{prodsetdef}
\bigcup_{n\in\N}V_1V_2\cdots V_n
\end{equation}
is an identity neighbourhood in $(G,\cO)$
for all identity neighbourhoods~$V_n$ in $(G_n,\cO_n)$.
If product sets are large in $(G,\cO)$, then $\cO=\cO_{TG}$ (see \cite[Proposition~11.8]{JFA}),
and moreover the product sets as in~(\ref{prodsetdef})
form a basis of identity neighbourhoods for $(G,\cO)$ (as we recall in Lemma~\ref{prodsets-basis}).
We mention that sets of the form
\[
\bigcup_{n\in\N} (V_n\cdots V_2V_1)(V_1V_2\cdots V_n)
\]
(so-called ``bamboo shoots") were already used in~\cite{Tat} and~\cite{Hir}
to obtain tangible descriptions of the topology $\cO_{TG}$
in well-behaved situations.
Our main result can be formulated as follows:\\[4mm]
{\bf Theorem A}
\emph{Let $G_1\sub G_2\sub\cdots$ be a strict direct sequence of topological groups
and $G:=\bigcup_{n\in\N}G_n$ be its direct limit topological group.}
\begin{itemize}
\item[(a)]
\emph{If product sets are large in~$G$,
then each inclusion map $G_n\to G$ is a topological embedding.}
\item[(b)]
\emph{If product sets are large in~$G$ and each $G_n$
is complete, then also $G$ is complete.}
\end{itemize} 
We mention that $(G,\cO_{TG})$ is Hausdorff if each $G_n$ is Hausdorff
and the inclusion map $G_n\to G$ a topological embedding.\footnote{If $e\not=x\in G_n$,
there is an open identity neighbourhood $V\sub G_n$ with $x\not\in V$.
Let $W\sub G$ be an open identity neighbourhood such that $W\cap G_n=V$.
Then $x\not\in W$. Hence $\wb{\{e\}}=\{e\}$ in~$G$ and thus $G$ is Hausdorff.} \\[2.3mm]
Theorem~A and its proof were inspired by Bourbaki's discussion
of completeness for strict direct limits of complete locally convex spaces
(see Propo\-sition~9 in \cite[Ch.\,II, \S4, no.\,6]{TVS}).\\[2.3mm]
Now consider a family $(G_j)_{j\in J}$ of Lie groups~$G_j$ modelled on locally convex spaces~$E_j$.
Then the so-called \emph{weak direct product}
\[
G:=\bigoplus_{j\in J}G_j:=\Big\{(x_j)_{j\in J}\in\prod_{j\in J} G_j\colon \,\mbox{$x_j=e$ for almost all~$j$}\big\}\vspace{-1mm}
\]
can be made a Lie group modelled on the locally convex direct sum $E:=\bigoplus_{j\in J}E_j$
in such a way that for some $C^\infty$-diffeomorphisms $\phi_j\colon U_j\to V_j$
from an open identity neighbourhood $U_j\sub G_j$ onto an open $0$-neighbourhood $V_j\sub E_j$
with $\phi_j(e)=0$, the set
\[
\bigoplus_{j\in J}U_j:=G\cap \prod_{j\in J} U_j\vspace{-1mm}
\]
is an open identity neighbourhood in~$G$ and the map
\[
\oplus_{j\in J} \phi_j\colon \bigoplus_{j\in J} U_j\to\bigoplus_{j\in J}V_j\sub E,\;\;
(x_j)_{j\in J}\mto (\phi_j(x_j))_{j\in J}\vspace{-1mm}
\]
is a $C^\infty$-diffeomorphism (see \cite{MEA}).
If~$J$ is countable, then the topological group underlying the weak direct product $\bigoplus_{j\in J}G_j$
is the small box product
of the topological groups $G_j$ (as in \cite{Bk1}).
If~$J$ is uncountable, then the weak direct product
and the small box product still coincide as groups,
but the box topology is coarser and can be properly coarser. For example,
this happens for the family $(\R)_{j\in J}$ for an uncountable set~$J$.
The weak direct product
\[
\R^{( J )}:=\bigoplus_{j\in J}\R\vspace{-1mm}
\]
then coincides with the locally convex direct sum, whose topology differs from the box topology,\footnote{In fact,
$\{(x_j)_{r\in J}\in\R^{( J )}\colon \sum_{j\in J}|x_j|<1\}$ is a $0$-neighbourhood in the
locally convex direct sum which cannot contain any box $\bigoplus_{j\in J}]{-q_j},q_j[$
with $q_j  \in \; ]0,\infty[\, \cap\, \Q=:C$ as one of the sets $J_q:=\{j\in J\colon q_j=q\}$
with $q \in C$ must be uncountable and hence infinite.}
as is well known (see, e.g., \cite{Wae}).
We show:\\[4mm]
{\bf Theorem B.}
\emph{Let $(G_j)_{j\in J}$ be a family of Lie groups $G_j$
modelled on locally convex spaces.
If each $G_j$ is complete} (\emph{resp., sequentially complete}),
\emph{then also the weak direct product $\bigoplus_{j\in J} G_j$
is complete} (\emph{resp., sequentially complete}).\\[4mm]
Similarly, one finds that the small box product of each family
of complete (resp., sequentially complete)
topological groups is complete (resp.,
sequentially complete), see Example~\ref{box-complete}.\\[4mm]
We now explain how the main results (and further findings)
can be used to establish completeness for infinite-dimensional Lie groups
within the main classes of examples (as listed in~\cite[pp.\,3-4]{Nee}),
and related topological groups.\\[2.5mm]
{\bf Direct limits of finite-dimensional Lie groups.}
If $G_1\sub G_2\sub\cdots$ is a direct sequence of topological groups
and the direct limit topology $\cO_{DL}$ on $G=\bigcup_{n\in\N}G_n$
makes~$G$ a topological group (i.e., if $\cO_{TG}=\cO_{DL}$),
then
product sets are large in $(G,\cO_{TG})$ (see \cite[Proposition 11.3]{FUN}).
Thus
Theorem~A entails:\\[2.3mm]
\emph{If $\cO_{TG}=\cO_{DL}$ on $G=\bigcup_{n\in\N}G_n$
for a strict direct sequence $G_1\sub G_2\sub\cdots$ of complete
topological groups, then $(G,\cO_{TG})$ is complete.}\\[2.3mm]
We mention that $\cO_{TG}=\cO_{DL}$ on $G=\bigcup_{n\in\N}G_n$
for each direct sequence $G_1\sub G_2\sub\cdots$
of locally compact Hausdorff topological groups (see \cite{Tat} and \cite{Hir}).
Hence every strict direct limit $G=\bigcup_{n\in\N} G_n$
of locally compact Hausdorff topological groups $G_1\sub G_2\sub \cdots$ is complete.
In particular, the Lie groups $\dl G_n$\vspace{-.9mm}
(as in~\cite{DL})
are complete for each strict direct sequence $G_1\sub G_2\sub\cdots$ of finite-dimensional
Lie groups.\footnote{The Lie group structure on important
examples of such groups (like $\GL_\infty(\R)=\dl \GL_n(\R)$
and $\SL_\infty(\R)=\dl \SL_n(\R)$) was already constructed in~\cite{KaM};
cf.\ also \cite{NRW}.}\\[3mm]
{\bf Diffeomorphism groups.}
For $M$ a paracompact finite-dimensional smooth manifold, consider the group
$\Diff_c(M)$ of all $C^\infty$-diffeomorphisms $\phi\colon M\to M$ with compact support (in the sense that
$\phi(x)=x$ for $x$ outside some compact set).
Then $\Diff_c(M)$ is a Lie group modelled on the space of smooth compactly supported
vector fields on~$M$ (and $\Diff(M)$ can be made a Lie group
with $\Diff_c(M)$ as an open normal subgroup), see \cite{Mic}
(or also \cite{PAT} and \cite{GaN} if~$M$ is $\sigma$-compact).
For each compact subset $K\sub M$,
\[
\Diff_K(M):=\{\phi\in \Diff_c(M)\colon( \forall x\in M\setminus K)\; \phi(x)=x\}
\]
is a Lie subgroup of $\Diff_c(M)$, modelled on the Fr\'{e}chet space
of all smooth vector fields supported in~$K$.
If $M$ is $\sigma$-compact and $K_1\sub K_2\sub\cdots$ an exhaustion of~$M$
by compact sets,\footnote{Thus $M=\bigcup_{n\in\N} K_n$ and $K_n\sub K_{n+1}^0$ for each $n\in\N$.} then
\begin{equation}\label{diffeoseq}
\Diff_{K_1}(M)\sub\Diff_{K_2}(M)\sub\cdots
\end{equation}
is a strict direct sequence of Lie groups.
By~\cite[Example 11.7]{JFA}, the map
\[
\pi\colon \bigoplus_{n\in\N} \Diff_{K_n}(M)\to\Diff_c(M)\vspace{-2mm}
\]
taking $(\phi_1,\ldots,\phi_n,\id_M,\id_M,\ldots)$ to $\phi_1\circ\cdots\circ\phi_n$
admits a smooth local section around~$\id_M$ (in the spirit of fragmentation
techniques familiar in the theory of diffeomorphism groups, cf.\ \cite{Bny}).
By \cite[Remark 11.5 and Proposition 11.8]{JFA}, this implies that product sets are large in~$\Diff_c(M)$
and $\Diff_c(M)$ is the direct limit topological group of~(\ref{diffeoseq}),
as recorded in \cite[Proposition 5.4]{JFA}
(see also Remark~1 and Proposition~1 in \cite{Bk4} for these arguments).\footnote{We mention that the topology on the Lie group $\Diff_c(M)$
coincides with the Whitney $C^\infty$-topology
used in~\cite{Bk4}; this is clear from the description
of this topology in~\cite{Ill} (see also \cite{HaS}
for a detailed account).}\\[2.3mm]
Now $\Diff_{K_n}(M)$ is a
strong (ILB)-Lie group (as considered in~\cite{Omo}) for each $n\in\N$. Using that strong (ILB)-Lie groups 
are complete (see Proposition~\ref{henceilb} and Remark~\ref{ilb2}\,(a)),
Theorem~A implies that
$\Diff_c(M)$ is complete for $\sigma$-compact~$M$
(see Remark~\ref{ilb2} (b) and (c) for details).\\[2.3mm]
If~$M$ is merely paracompact and $(M_j)_{j\in J}$
its family of connected components, then~$\Diff_c(M)$
has an open subgroup~$G$ which is isomorphic to
the weak direct product
$\bigoplus_{j\in J}\Diff_c(M_j)$
as a Lie group, and we deduce with Theorem~B that $G$ (and hence $\Diff_c(M)$)
is complete also for paracompact~$M$.\\[2.3mm]
The completeness of diffeomorphism groups contrasts the
incompleteness of many groups of homeomorphisms,
among which one finds typical examples of metrizable topological
groups which cannot be completed as the sets of Cauchy sequences
for the left and right uniformity do not coincide (see~\cite{Die}).\\[3mm]
{\bf Mapping groups and gauge groups.}
Among the prime examples of infinite-dimensional
Lie groups are the Lie groups $C^k(M,H)$ of $C^k$-maps
from a compact manifold~$M$ to a Lie group~$H$ for $k\in\N_0\cup\{\infty\}$
(notably the \emph{loop groups} with $M=\bS^1$ the circle
group \cite{PaS}), see \cite{Mil} and \cite{Omo}.
More generally,
if $M$ is a paracompact finite-dimensional smooth manifold
and~$H$ a Lie group modelled on a locally convex space~$E$,
there is a natural Lie group structure on the group $C^k_c(M,H)$
of all $C^k$-maps $\gamma \colon M\to H$ whose support
\[
\Supp(\gamma):=\overline{\{x\in M\colon \gamma(x)\not=e\}}
\]
is compact (where $e$ is the neutral element of~$H$),
which is modelled on the locally convex space $C^k_c(M,E)$;
see \cite{FUN} (also \cite{Alb}) if $M$ is $\sigma$-compact;
in the general case, we let $(M_j)_{j\in J}$ be the family of connected components
of~$M$ and use the group isomorphism
\[
C^k_c(M,H)\to\bigoplus_{j\in J}C^k_c(M_j,H),\quad
\gamma\mto(\gamma|_{M_j})_{j\in j}\vspace{-1mm}
\]
to transport the Lie group structure of the weak direct product
to the left hand side.
Using Theorems~A and~B, we shall see that $C^k_c(M,H)$ is complete
whenever $H$ and its modelling space~$E$ are complete
(Proposition~\ref{mapscomplete}).
Likewise, gauge groups
and full symmetry groups of principal bundles
(as considered in \cite{Sch} and~\cite{Woc})
are complete if the structure group~$H$ and its
modelling space are complete (see Remark~\ref{alsogauge}
for more details).\\[3mm]
{\bf Linear Lie groups.}
We can also prove completeness for some unit groups
of topological algebras.\footnote{If $A$ is a unital algebra,
we write $A^\times$ for its \emph{unit group} of all invertible elements.}
Consider an ascending sequence $A_1\sub A_2\sub\cdots$
of unital Banach algebras, such that all inclusion maps $A_n\to A_{n+1}$
are continuous homomorphisms of unital algebras.
Endow $A:=\bigcup_{n\in\N}A_n$ with the unital algebra structure
turning each inclusion map $A_n\to A$ into a homomorphism of unital algebras.
Then the locally convex direct limit topology makes~$A$ a topological
algebra and product sets are large in $A^\times=\bigcup_{n\in\N}A_n^\times$
(see \cite[Proposition~12.1 (a) and (c)]{JFA}).\footnote{Compare also \cite{DaW}
for general information on such algebras.}
With Theorem~A, we deduce that $A^\times$ (like each $A_n^\times$)
is complete whenever the direct sequence $A_1\sub A_2\sub\cdots$ is strict.\\[3mm]
{\bf Ascending unions of Banach-Lie groups.}
Beyond unit groups of Banach algebras,
let us consider an ascending sequence
$G_1\sub G_2\sub\cdots$ of Banach-Lie groups over $\K\in\{\R,\C\}$
such that each inclusion map $G_n\to G_{n+1}$ is a $\K$-analytic group
homomorphism. In \cite[Theorem~C]{Da1},
conditions were spelled out which ensure that $G=\bigcup_{n\in\N}G_n$
can be made a $\K$-analytic Lie group modelled on the locally convex direct limit
of the respective Lie algebras.
We show that product sets are large in the Lie groups~$G$
constructed in loc.\ cit., whence the given topology on~$G$ 
makes it the direct limit topological group $G=\dl G_n$ (see Proposition~\ref{arelarge}).\vspace{-.8mm}
Using our Theorem~A, we deduce that~$G=\bigcup_{n\in\N}G_n$ (as before)
is complete whenever the direct sequence
$G_1\sub G_2\sub\cdots$ is strict (see Proposition~\ref{arelarge}).\\[2.3mm]
For a concrete example, let $(F_n,\|.\|_n)_{n\in\N}$
be a sequence of Banach spaces. Write $\cL(F)$
for the Banach algebra of bounded operators
$S\colon F\to F$ for a Banach space~$F$, endowed with the operator norm~$\|.\|_{op}$.
We equip $E_n:=F_1\oplus\cdots\oplus F_n$
with the maximum norm and identify
\[
\GL(E_n):=\cL(E_n)^\times
\]
with the subgroup $\GL(E_n)\times\{\id_{F_{n+1}}\}$ of $\GL(E_{n+1})$.
Using \cite[Theorem~A]{Dah},
it was shown in \cite[Theorem~39]{CaP}
that
\[
\GL((F_n)_{n\in\N}):=\bigcup_{n\in\N}\GL(E_n)\vspace{-1mm}
\]
can be made a Lie group.\footnote{In loc.\ cit., $\GL((F_n)_{n\in\N})$ is denoted by $\GL(E)$,
with $E:=\bigcup_{n\in\N}E_n$.} As the direct sequence $\GL(E_1)\sub \GL(E_2)\sub\cdots$
is strict, the preceding reasoning shows that $\GL((F_n)_{n\in\N})$ is complete.\\[2.3mm]
{\bf Outlook: Lie groups modelled on Silva spaces.}
Work by Hunt and Morris~\cite{HaM}
implies that every Lie group modelled on a Silva space\footnote{A locally convex space is called a \emph{Silva space}
or (DFS)-space if it is a locally convex direct limit $\dl\, E_n$\vspace{-.7mm}
for an ascending sequence $E_1\sub E_2\sub\cdots$
of Banach spaces, such that all inclusion maps $E_n\to E_{n+1}$
are compact operators.}
is complete (see \cite[Corollary~1.4]{BaG}). This entails that direct limits $\bigcup_{n\in\N}G_n$
of finite-dimensional Lie groups (or locally compact groups) $G_1\sub G_2\sub\cdots$ are always complete
(no matter whether the direct sequence is strict or not).
It also shows that the Lie group $\Diff^\omega(M)$
of real-analytic diffeomorphisms is complete for each compact real-analytic manifold~$M$,
as well as the Lie group $C^\omega(M,H)$ of all real-analytic $H$-valued mappings
on the latter, for each finite-dimensional Lie group~$H$ (see \cite{BaG}).\\[3mm]
{\bf Acknowledgement.}
I am grateful to my former Ph.D.-students
Rafael Dahmen und Boris Walter
who sketched a precursor of Lemma~\ref{sub-complete}
(which feeds into the proof of Theorem~B)
in a seminar talk in~2009.
\section{Preliminaries and notation}\label{sec1}
We write $\N=\{1,2,\ldots\}$ and $\N_0:=\N\cup\{0\}$.
Topological groups and locally convex (real topological vector) spaces
shall not be assumed Hausdorff, unless we say so explicitly.
If $f\colon X\to Y$ is a function between metric spaces $(X,d_X)$ and $(Y, d_Y)$,
we define
\[
\Lip(f):=\sup\left\{\frac{d_Y(f(x),f(y))}{d_X(x,y)}\colon x\not=y\in X\right\}\in
[0,\infty]
\]
and call $f$ \emph{Lipschitz} if $\Lip(f)<\infty$.
If each point $x\in X$ has a neighbourhood $V\sub X$ such that $f|_V\colon V\to Y$
is Lipschitz (with respect to the metric $d_X|_{V\times V}$ induced on~$V$),
then $f$ is called \emph{locally Lipschitz}.
If $(E,\|.\|)$ is a Banach space,
we write $\GL(E)$ for the group of continuous automorphisms of
the vector space~$E$.
For $x\in E$ and $r>0$, we write $B^E_r(x):=\{y\in E\colon \|y-x\|<r\}$
and $\wb{B}^E_r(x):=\{y\in E\colon \|y-x\|\leq r\}$.
If $q$ is a continuous seminorm on a locally convex space~$E$,
we write $\wb{B}^q_r(0):=\{x\in E\colon q(x)\leq r\}$
for $r>0$.
%
\begin{numba}
For our setting of $C^k$-maps, real and complex
analytic mappings between open subsets of locally
convex Hausdorff topological vector spaces, the reader is referred to
\cite{RES}, \cite{Nee}, and \cite{GaN},
where also the corresponding concepts of manifolds
(and Lie groups) modelled on Hausdorff locally convex spaces
are described
(cf.\ also \cite{BGN}, \cite{Ham}, \cite{Mic}, and \cite{Mil}).\footnote{The $C^k$-maps
are those of Keller's $C^k_c$-theory~\cite{Kel}, going back to A. Bastiani \cite{Bas}.}
All of these manifolds (and Lie groups)
are Hausdorff.
For $C^k$-manifolds with boundary modelled on Hausdorff locally
convex spaces, see~\cite{GaN}.
Every Lie group~$G$ is a topological group
in its given topology (as smooth maps are continuous
in the infinite-dimensional calculus we are using).
We write $e$ for the neutral element of~$G$
and $L(G):=T_e(G)$ (or simply $\cg$)
for its Lie algebra (the tangent space at~$e$). We write $\exp_G\colon L(G)\to G$
for the exponential function of~$G$ (if it exists), as in~\cite{Nee}
or~\cite{GaN} (cf.\ also \cite{Mil}).
If $f\colon G\to H$
is a smooth group homomorphism between
Lie groups, we abbreviate $L(f):=T_e(f)$.
If $M$ is a $C^k$-manifold with $k\geq 1$,
we let $\pi_{TM}\colon TM\to M$ be the bundle projection.
If $E$ is a locally convex space and $U\sub E$ an open subset,
we identify the tangent bundle $TU$ with $U\times E$ in the usual way.
If $M$ is a $C^1$-manifold and $f\colon M\to U$ a $C^1$-map,
we write $df\colon TM\to E$ for the second component of
$Tf\colon TM\to TU=U\times E$. If $g\colon U\to F$ is a $C^1$-map to a locally
convex space, we write $g'(x):=dg(x,\cdot)\colon E\to F$
for $x\in U$.
\end{numba}
The following fact is based on the
Mean Value Theorem.
\begin{numba}\label{C1lip}
(cf.\ \cite[Lemma~2.3.16]{GaN}).
Let $(E,\|.\|_E)$ and $(F,\|.\|_F)$ be normed spaces,
$U\sub E$ be open and $f\colon U\to F$
be a $C^1$-map. Then the following holds:
\begin{itemize}
\item[\rm(a)]
$f$ is locally Lipschitz.
\item[\rm(b)]
If $U$ is convex, then $f$ is Lipschitz if and only
if $L:=\sup_{x\in U}\|f'(x)\|_{op}<\infty$,
in which case $\Lip(f)=L$.
\end{itemize}
\end{numba}
Recall that a net $(x_\alpha)_{\alpha\in A}$ in a topological
group~$G$ is called a left Cauchy net if, for each identity neighbourhood $U\sub G$,
there exists $\gamma\in A$ such that
\[
x_\beta^{-1}x_\alpha \in U\quad\mbox{for all $\,\alpha,\beta\geq \gamma$.}
\]
If every left Cauchy net in~$G$ is convergent, then $G$ is called \emph{complete};
if every left Cauchy sequence in~$G$ is convergent, then $G$ is
\emph{sequentially complete}.
\begin{numba}\label{factscompl}
Many results concerning completeness of topological groups can be found in~\cite{RaD}.
We mention useful facts:
\begin{itemize}
\item[(a)]
If a topological group~$G$ is complete (resp., sequentially complete),
then every closed subgroup $H\sub G$ is complete in the induced topology.
\item[(b)]
For every family $(G_j)_{j\in J}$ of topological groups which are complete (resp.,
sequentially complete), the direct product $\prod_{j\in J} G_j$
is complete (resp., sequentially complete) in the product topology.
\item[(c)]
Let $((G_j)_{j\in J},(q_{i,j})_{i\leq j})$ be a projective system
of Hausdorff topological groups $G_j$ and continuous homomorphisms $q_{i,j}\colon G_j\to G_i$
for $i\leq j$ in~$J$ such that $q_{i,i}=\id_{G_i}$ and $q_{i,j}\circ q_{j,k}=q_{i,k}$
whenever $i\leq j\leq k$.
If each $G_i$ is complete (resp., sequentially complete),
then also the projective limit $\pl \, G_j$ is complete (resp., sequentially complete),
as it can be realized as the closed subgroup
\[
\Big\{(x_j)_{j\in J}\in \prod_{j\in J} G_j\colon (\forall i\leq j)\; x_i=q_{i,j}(x_j)\Big\}\vspace{-1mm}
\]
of the direct product, endowed with the induced topology.
\item[(d)]
Completeness is an extension property: If $G$ is a topological
group and $N\sub G$ a normal subgroup such that both~$N$ and $G/N$
are complete, then also~$G$ is complete (cf.\ \cite[Theorem~12.3\,(a)]{RaD}).
\end{itemize}
If $G$ is a topological group, $H\sub G$ a subgroup and $N\sub G$ a normal subgroup,
we say that $G$ is the (internal) \emph{semidirect product} of $N$ and $H$
as a topological group if the product map
$N\times H\to G$, $(x,y)\mto xy$
is a homeomorphism. Since $q\colon G\to H$, $xy\mto y$ is a quotient homomorphism
with kernel~$N$, the following holds as special case of~(d):
\begin{itemize}
\item[(e)]
Let $G$ be a topological group which, as a topological group,
is the internal semidirect product of a normal
subgroup $N$ and a subgroup~$H$. If~$N$ and~$H$ are complete,
then also~$G$ is complete (\cite[Proposition~12.5\,(a)]{RaD}).
\end{itemize}
\end{numba}
The following slight generalization of \ref{factscompl}\,(c)
is useful.
\begin{la}\label{slightgen}
Let $G$ be a topological group whose underlying topological space is the projective limit
of a projective system $((X_j)_{j\in J},(q_{i,j})_{i\leq j})$
of Hausdorff topological spaces~$X_j$, with limit maps $q_j\colon G\to X_j$
for $j\in J$. Assume that for each Cauchy net $(x_\alpha)_{\alpha\in A}$ in~$G$,
the corresponding net $(q_j(x_\alpha))_{\alpha\in A}$ converges in~$X_j$,
for each $j\in J$. Then $G$ is complete.
\end{la}
\begin{proof}
We may assume that
\[
G=\Big\{(x_j)_{j\in J}\in \prod_{j\in J}X_j\colon (\forall i\leq j)\; x_i=q_{i,j}(x_j)\Big\}
\]
and $q_i((x_j)_{j\in J})=x_i$ for all $i\in I$ and $x=(x_j)_{j\in J}\in G$.
For $j\in J$, let $y_j\in X_j$ be the limit of $(q_j(x_\alpha))_{\alpha\in A}$.
For $i,j\in I$ with $i\leq j$, the net
\[
(q_i(x_\alpha))_{\alpha\in A}=(q_{i,j}(q_j(x_\alpha)))_{\alpha\in A}
\]
in~$X_i$ converges both to $y_i$ and $q_{i,j}(y_j)$.
Since~$X_i$ is Hausdorff, $y_i=q_{i,j}(y_j)$ follows.
Hence $y:=(y_j)_{j\in J}\in G$ and as $q_j(x_\alpha)\to y_j=q_j(y)$ for all $j\in J$,
the net $(x_\alpha)_{\alpha\in A}$ converges to~$y$.
\end{proof}
\begin{numba}\label{smallbox}
If $(G_j)_{j\in J}$ is a family of topological
groups, let
\[
\square_{j\in J}G_j\sub\prod_{j\in J} G_j\vspace{-1mm}
\]
be the subgroup of all $(x_j)_{j\in J}\in\prod_{j\in J} G_j$
such that $x_j=e$ for all but finitely many~$j$.
Consider the sets
\[
\square_{j\in J} U_j\; :=\; \prod_{j\in J} U_j\, \cap\,  \square_{j\in J} G_j,
\]
for $(U_j)_{j\in J}$ ranging through the families of open subsets $U_j\sub G_j$
such that $e\in U_j$ for all but finitely many~$j$.
The latter sets form a basis for a topology on $\square_{j\in J} G_j$
making it a topological group (called the \emph{box topology}).
When endowed with this topology, $\square_{j\in J} G_j$ is called the \emph{small box product}
of the family $(G_j)_{j\in J}$
(see \cite{Bk1} for further details).\\[2.3mm]
If each $G_j$ is a Lie group modelled on a Hausdorff locally convex space,
then also $\square_{j\in J}G_j$ is a Lie group
in a natural way (see \cite{ZOO});\footnote{In \cite{ZOO},
small box products are called weak direct products,
in contrast to the conventions in the current article.}
it is modelled on the small box product
$\square_{j\in J}E_j$.
If, instead, we use the locally convex direct sum as the modelling space,
then the group $\square_{j\in J}G_j$ can be made a Lie group
as well, called the \emph{weak direct product} of the family
and denoted by $\bigoplus_{j\in J} G_j$ in this article
(see \cite{MEA}, where the notation $\prod_{j\in J}^*G_j$ is used).
The two possible modelling spaces (and the two Lie groups)
coincide if~$J$ is countable. When dealing with $\bigoplus_{j\in J}G_j$,
we write $\bigoplus_{j\in J}U_j$ instead of $\square_{j\in J}U_j$.
\end{numba}
\section{Deducing completeness from completeness of a subgroup}
The following lemma is essential for the proof of Theorem~A.
\begin{la}\label{zweimal}
Let $(x_\alpha)_{\alpha\in A}$ be a left Cauchy net in a topological group~$G$
and $H\sub G$ be a subgroup which is a complete topological group
in the induced topology.
Assume that, for each $\alpha\in A$ and identity neighbourhood $W\sub G$,
there exists $\beta\geq\alpha$ such that
\begin{equation}\label{convmodu}
x_\beta\in H W.
\end{equation}
Then $(x_\alpha)_{\alpha\in A}$ converges in~$G$, to some~$y\in H$.
\end{la}
\begin{proof}
Let~$\cU$ be the set of all identity neighborhoods in~$G$.
By hypothesis,
\[
A_W:=\{\alpha\in A\colon x_\alpha\in H W\}
\]
is cofinal in~$A$ for all $W\in\cU$ and thus
$M:=\{(W,\alpha)\in\cU\times A\colon \alpha\in A_W\}$
becomes a directed set if we write $(W_1,\alpha_1)\leq (W_2,\alpha_2)$ if and only if
$W_2\sub W_1$ and $\alpha_1\leq \alpha_2$.
For $a=(W,\alpha)\in M$, pick $y_a\in H $ and $w_a\in W$ such that
\begin{equation}\label{thusconvsu}
x_\alpha=y_a w_a.
\end{equation}
Then $(y_a)_{a\in M}$ is a left Cauchy net in~$H$.
In fact, if $U$ is an identity neighborhood in~$H$, we find $Q\in\cU$ such that $U= H\cap Q$.
Let $P\in\cU$ such that $PPP^{-1}\sub Q$ and $\gamma\in A$ such that
$x_\beta^{-1}x_\alpha \in P$ for all $\alpha,\beta\geq \gamma$.
We may assume that $\gamma\in A_P$. For all $a,b\geq (P,\gamma)$ in~$M$, say $a=(W,\alpha)$ and $b=(V,\beta)$,
we have $V,W\sub P$ and hence
\[
y_b^{-1}y_a= w_b x_\beta^{-1} x_\alpha w_a^{-1} \in V P W^{-1} \sub PPP^{-1}\sub Q.
\]
Thus $y_b^{-1}y_a\in H\cap Q=U$.\\[2.3mm]
Let $y$ be the limit of $(y_a)_{a\in M}$ in~$H$.
Then $y_a\to y$ also in~$G$. Given $W\in\cU$, let $\alpha\in A_W$.
Since $w_a\in W$ for $a\geq (W,\alpha)$, the net $(w_a)_{a\in M}$ converges to~$e$ in~$G$.
Using (\ref{thusconvsu}), we deduce that the subnet
$(x_\alpha)_{(W,\alpha)\in M}$
of $(x_\alpha)_{\alpha\in A}$
(and hence also the Cauchy net $(x_\alpha)_{\alpha\in A}$) converges to~$y$.
\end{proof}
\section{Completeness of strict direct limits}\label{proof-A}
In this section, we prove Theorem~A.
\begin{la}\label{prepare-prod}
Let $G$ be a group, $2\leq n\in\N$ and $G_1\sub G_2\sub \cdots\sub G_n=G$
be subgroups. For $j\in \{1,\ldots, n\}$, let $W_j$ be a subset of~$G_j$.
Then
\begin{equation}\label{redulast}
G_1\cap (W_1 W_2\cdots W_n)=G_1\cap (W_1\cdots W_{n-1}(G_{n-1}\cap W_n)).
\end{equation}
\end{la}
\begin{proof}
We show by induction on~$n$ that the left hand side of (\ref{redulast})
is a subset of the right hand side (the other inclusion is trivial),
for all $G$, $G_1,\ldots, G_n$ and $W_1,\ldots,W_n$
as described in the lemma.
If $n=2$ and $x\in G_1\cap W_1W_2$,
then $x=w_1w_2$ with $w_1\in W_1$ and $w_2\in W_2$. Since $W_1\sub G_1$, we have $w_1\in G_1$
and thus $w_2=w_1^{-1}x\in G_1\cap W_2$.

If $n>2$ and the assertion holds for $n-1$, let $x\in G_1\cap (W_1\cdots W_n)$.
Write $x=w_1w_2\cdots w_n$ with $w_j\in W_j$ for $j\in\{1,\ldots, n\}$.
Then $w_2\cdots w_n=w_1^{-1}x\in G_1\cap (W_2\cdots W_n)\sub G_2\cap (W_2\cdots W_n)$
and thus $w_n\in G_{n-1}\cap W_n$, by the inductive hypothesis.
\end{proof}
\begin{la}\label{prodsets-basis}
Assume that $G$ is the direct limit topological group of a direct
sequence $G_1\sub G_2\sub\cdots$ of topological groups
and product sets are large in~$G$.
Then every identity neighbourhood of~$G$ contains a product set
$\bigcup_{n\in\N}W_1W_2\cdots W_n$
for suitable identity neighbourhoods $W_n\sub G_n$.
\end{la}
\begin{proof}
If $V_0$ is an identity neighbourhood in~$G$,
there exist identity neighbourhoods $V_n\sub G$ for $n\in\N$
such that $V_n V_n\sub V_{n-1}$.
Then $W_n:=G_n\cap V_n$ is an identity neighbourhood in~$G_n$ and
$V_1V_2\cdots V_n\sub V_0$ for all $n\in\N$
implies that $\bigcup_{n\in\N}W_1\cdots W_n\sub \bigcup_{n\in\N}V_1\cdots V_n\sub V_0$.
\end{proof}
{\bf Proof of Theorem~A.}
(a) To see that $G$ induces the given topology on~$G_n$, we may assume that $n=1$.
Let $V_1\sub G_1$ be an identity neighbourhood.
There exists an identity neighbourhood $W_1\sub G_1$ such that $W_1W_1\sub V_1$.
Recursively, for $m\geq 2$ find an identity neighbourhood $V_m\sub G_m$
such that
\[
G_{m-1}\cap V_m=W_{m-1}
\]
(which is possible as $G_m$
induces the given topology on~$G_{m-1}$) and an identity neighbourhood $W_m\sub G_m$
such that
$W_mW_m\sub V_m$.
Then
\begin{equation}\label{finite-cas}
G_1\cap (W_1W_2\cdots W_m)\sub G_1\cap (W_1\cdots W_{j-1}V_j)
\end{equation}
for all $j\in\{m,m-1,\ldots,1\}$, by induction:
If $j=m$, we have $W_m\sub V_m$ and the assertion holds.
If $j\in\{2,\ldots, m\}$ and the assertion holds for~$j$,
then
\begin{eqnarray*}
G_1\cap (W_1W_2\cdots W_m) &\sub & G_1\cap (W_1\cdots W_{j-1}V_j)\\
&=&G_1\cap (W_1\cdots W_{j-1}(G_{j-1}\cap V_j))\\
&=& G_1\cap (W_1\cdots W_{j-1}W_{j-1})\\
&\sub&  G_1\cap (W_1\cdots W_{j-2} V_{j-1})
\end{eqnarray*}
using the inductive hypothesis, Lemma~\ref{prepare-prod},
the identity $G_{j-1}\cap V_j=W_{j-1}$
and the inclusion $W_{j-1}W_{j-1}\sub V_{j-1}$.
Taking $j=1$, we deduce that
\[
G_1\cap (W_1W_2\cdots W_m)\sub V_1
\]
for all $m\in \N$ and hence $G_1\cap W\sub V_1$ if we define
\[
W:=\bigcup_{m\in\N} W_1W_2\cdots W_m.\vspace{-1mm}
\]
As we assume that product sets are large in~$G$,
the set $W$ is an identity neighborhood in~$G$. Since $G_1\cap W\sub V_1$,
the group topology $\cT$ induced by $G$ on~$G_1$
is finer than the given topology $\cO_1$ on~$G_1$ and hence coincides with it
(noting that $\cT\sub \cO_1$ as the inclusion map $(G_1,\cO_1)\to G$
is continuous).

(b) If each $G_n$ is complete, let $(x_\alpha)_{\alpha\in A}$
be a left Cauchy net in~$G$. Let~$\cU$ be the set of all identity neighborhoods in~$G$.
We claim that there exists $m \in\N$
such that, for each $\alpha\in A$ and $W\in\cU$,
there exists $\beta\geq\alpha$ such that
\[
x_\beta\in G_m W.
\]
If this is true, then $(x_\alpha)_{\alpha\in A}$ converges in~$G$,
by Lemma~\ref{zweimal}, using that~$G$ induces the given complete group
topology on~$G_m$, by~(a).

To prove the claim, suppose it was wrong.
Then, for each $m\in\N$,
there exist $\alpha_m\in A$ and $W_m\in\cU$
such that
\begin{equation}\label{no-inters}
x_\alpha\not \in  G_m W_m\quad\mbox{for all $\alpha\geq\alpha_m$.}
\end{equation}
After shrinking $W_m$ if necessary, we may assume that
\[
W_m=\bigcup_{n\in\N} W^{(m)}_1\cdots W^{(m)}_n
\]
with identity neighbourhoods $W^{(m)}_n\sub G_n$,
by Lemma~\ref{prodsets-basis}.
After shrinking $W^{(2)}_n\!$, $W^{(3)}_n\!$, $\ldots\;$, we may assume that
\begin{equation}\label{geschachtelt}
W^{(m+1)}_n\sub W^{(m)}_n\quad\mbox{for all $\,n,m\in\N$.}
\end{equation}
Since product sets are large in~$G$, the set
$W:=\bigcup_{n\in\N}W^{(1)}_1\cdots W^{(n)}_n$
is an identity neighbourhood in~$G$.
By (\ref{geschachtelt}), we have
\[
\bigcup_{n>m}W^{(m+1)}_{m+1}\cdots W^{(n)}_n\sub \bigcup_{n>m}W^{(m)}_{m+1}\cdots W^{(m)}_n
\sub\bigcup_{n\in\N}W^{(m)}_1\cdots W^{(m)}_n=W_m.
\]
Using that $W^{(1)}_1\cdots W^{(n)}_n\sub G_m$ for $n\in\{1,\ldots, m\}$, we deduce that
\[
G_m W =G_m\bigcup_{n>m}W^{(m+1)}_{m+1}\cdots W^{(n)}_n\sub G_m W_m
\]
and thus
\begin{equation}\label{allW}
G_m W\sub G_m W_m\quad\mbox{for all $\,m\in \N$.}
\end{equation}
By definition of a Cauchy net, we find $\gamma\in A$ such that
\begin{equation}\label{diff-in}
x_\alpha^{-1}x_\beta\in W
\end{equation}
for all $\alpha,\beta\geq \gamma$. Now $x_\gamma\in G_{m_0}$ for some $m_0\in\N$.
Since $A$ is directed, we find $\alpha\in A$ such that $\alpha\geq \gamma$ and $\alpha\geq\alpha_{m_0}$.
Using (\ref{diff-in}), we obtain
\[
x_\alpha =x_\gamma(x_\gamma^{-1}x_\alpha)\in G_{m_0}W.
\]
But $x_\alpha\not\in G_{m_0} W_{m_0}$ by (\ref{no-inters}) and thus $x_\alpha\not\in G_{m_0}W$ (by (\ref{allW})),
which is absurd.\Punkt
\section{Completeness of weak direct products}\label{proof-B}
The following lemma will enable
us to reduce the completeness of
weak direct products (and box products) to that of direct products.
\begin{la}\label{sub-complete}
Let $P$ be a complete $($resp., sequentially complete$)$ topological group
and $G$ be a subgroup of~$P$, endowed with a topology~$\cO$
which is finer than the topology~$\cT$ induced by~$P$ on~$G$.
Assume that, for each $x\in P$ such that $x\not\in G$, 
there exists a closed subset $L$ in~$P$ such that
\begin{itemize}
\item[\rm(i)]
 $G\cap L $ is an identity neighborhood in~$(G,\cO)$; and
\item[\rm(ii)]
$G\cap xL =\emptyset$.
\end{itemize}
Moreover, assume that
\begin{itemize}
\item[\rm(iii)]
The closures $\wb{V}$ in~$(G,\cT)$ form a basis of identity neighbourhoods in $(G,\cO)$,
for $V$ in the set of identity neighbourhoods in~$(G,\cO)$.
\end{itemize}
Then also $(G,\cO)$ is complete $($resp., sequentially complete$)$.
\end{la}
\begin{proof}
Assume that $P$ is complete.
If $(x_\alpha)_{\alpha\in A}$ is a left Cauchy net in~$(G,\cO)$,
then it also is a left Cauchy net in~$P$, whence
$x_\alpha\to x$ in~$P$
for some $x\in P$. Then $x\in G$, since otherwise we obtain a contradiction:
Let $L\sub P$ be a closed subset such that $G\cap L$ is an identity neighbourhood
in $(G,\cO)$
and $G\cap xL =\emptyset$.
Let $\alpha_0\in A$ such that
$x_\alpha^{-1} x_\beta\in G\cap L$
for all $\alpha,\beta\geq\alpha_0$.
Considering $x_\alpha^{-1}x_\beta$ as elements of~$P$ and passing to the limit in~$\alpha$,
we obtain
$x^{-1}x_\beta\in L$
for all $\beta\geq \alpha_0$, whence $x_\beta\in G\cap xL=\emptyset$, which is absurd.

Let~$W$ be an identity neighborhood in~$(G,\cO)$. By hypothesis~(iii),
we find an identity neighbourhood~$V$ in $(G,\cO)$
such that
$G\cap \wb{V}\sub W$,
where $\wb{V}$ is the closure of~$V$ in~$P$.
There exists $\alpha_0\in A$ such that
$x_\alpha^{-1}x_\beta\in V$ for all $\alpha,\beta\geq\alpha_0$.
Considering $x_\alpha^{-1}x_\beta$ as an element of $P$ and passing to the limit in~$\alpha$,
we deduce that
$x^{-1}x_\beta\in\wb{V}$
for all $\beta\geq\alpha_0$, whence $x^{-1}x_\beta\in G\cap \wb{V}\sub W$
and thus $x_\beta\in xW$. Thus $x_\beta\to x$ in~$(G,\cO)$.

If $P$ is sequentially complete, then the proof is identical
with $A:=\N$.
\end{proof}
\begin{exa}\label{box-complete}
Let $(G_j)_{j\in J}$ be a family of topological groups $G_j$ which are complete
(resp., sequentially complete). Then also the small box product
$G:=\square_{j\in J} G_j$ is complete (resp., sequentially complete).\\[2.3mm]
To see this, let us check Lemma~\ref{sub-complete} can be used
with $P:=\prod_{j\in J}G_j$.
If $x=(x_j)_{j\in J}\in P\setminus G$,
then
\[
I:=\{j\in J\colon x_j\not=e\}
\]
is an infinite set. For each $j\in I$, let $U_j\sub G_j$ be a closed identity neigbourhood
such that $x_j^{-1}\not\in U_j$. For $j\in J\setminus I$, let $U_j:=G_j$.
Then $L:=\prod_{j\in J}U_j$ is closed in~$P$
and $G\cap L=\square_{j\in J}U_j$ is an identity neighbourhood such that $xL\cap G=\emptyset$
as each $y=(y_j)_{j\in J}\in xL$ satisfies $y_j\not=e$ for all $j\in I$.
Finally, let $W\sub G$ be an identity neighbourhood. Then $W$ contains a box
$V:=\square_{j\in J}V_j$ with closed identity neighbourhoods $V_j\sub G_j$.
The closure~$\wb{V}$ of~$V$ in~$P$ is $\prod_{j\in J}V_j$, whence $G\cap \wb{V}=V\sub W$.
\end{exa}
{\bf Proof of Theorem~B.}
To verify the theorem, let us check that Lemma~\ref{sub-complete} can be applied
with $P:=\prod_{j\in J}G_j$.
If $x=(x_j)_{j\in J}\in P\setminus G$,
then
\[
I:=\{j\in J\colon x_j\not=e\}
\]
is an infinite set. For each $j\in I$, let $U_j\sub G_j$ be a closed identity neigbourhood
such that $x_j^{-1}\not\in U_j$. For $j\in J\setminus I$, let $U_j:=G_j$.
Then $L:=\prod_{j\in J}U_j$ is closed in~$P$
and $G\cap L=\square_{j\in J}U_j$ is an identity neighbourhood (as the topology on~$G$
is finer than the box topology)
such that $xL\cap G=\emptyset$.
Thus~$L$ satisfies the conditions~(i) and~(ii)
in Lemma~\ref{sub-complete}.\\[2.3mm]
Next, let $S\sub G$ be an identity neighbourhood.
For $j\in J$, let $E_j$ be the locally convex space on which~$G_j$ is modelled.
Then $S$ contains an identity neighbourhood of the form
$W=\phi^{-1}(Q)$
for a diffeomorphism $\phi\colon U\to V$ and a $0$-neighbourhood $Q\sub V$,
where diffeomorphisms
$\phi_j\colon U_j \to V_j$ from open identity neighbourhoods
$U_j\sub G_j$ onto open $0$-neighbourhoods $V_j\sub E_j$
are used to define $U:=\bigoplus_{j\in J}U_j$, $V:=\bigoplus_{j\in J}V_j$ and
$\phi:=\bigoplus_{j\in J}\phi_j\colon U\to V$.
For each $j\in J$, the topological group $G_j$ has a closed identity neighbourhood $K_j$
such that $K_j\sub U_j$. Set $L_j:=\phi_j(K_j)$.
After shrinking~$Q$ (and $W=\phi^{-1}(Q)$), we may assume that
\[
W\sub \bigoplus_{j\in J} K_j=:K
\]
and thus $Q\sub\prod_{j\in J} L_j=:L$.
After shrinking~$Q$ further if necessary, we may also assume that
$Q=\wb{B}^q_r(0)$
for a continuous seminorm $q$ on $\bigoplus_{j\in J}E_j$,
and we may assume that~$q$ is of the form
\[
q(x)=\sum_{j\in J}q_j(x_j)\;\; \mbox{for all $\,x=(x_j)_{j\in J}\in \bigoplus_{j\in J}E_j$,}\vspace{-1mm}
\]
for certain continuous seminorms $q_j$ on~$E_j$.
Then the closure $\wb{Q}$ of~$Q$ in $\prod_{j\in J} E_j$
is the set~$C$  of all $(x_j)_{j\in J}\in\prod_{j\in J} E_j$
such that
\[
\sum_{j\in J}q_j(x_j)\leq r\vspace{-1mm}
\]
(where the sum means the supremum of all finite partial sums).
To see this, let $\Phi$ be the set of finite subsets of~$J$.
If $x\in C$, then
$\sum_{j\in F}x_j\in Q$ for each $F\in\Phi$ (as $\sum_{j\in F}q_j(x_j)\leq \sum_{j\in J}q_j(x_j)\leq r$)
and thus $x\in \wb{Q}$.\\[2.3mm]
Conversely, let $x=(x_j)_{j\in J}\in \prod_{j\in J} E_j$
with $x\not\in C$; thus $\sum_{j\in J}q_j(x_j)>r$.
There exists a finite subset $F\sub J$ such that
\[
\sum_{j\in F}q_j(x_j)>r.\vspace{-1mm}
\]
Since $\{(y_j)_{j\in J}\in \prod_{j\in J}E_j\colon \sum_{j\in F}q_j(y_j)>r\}$
is an open subset of $\prod_{j\in J}E_j$
which has empty intersection with~$Q$, we have
$x\not\in\wb{Q}$. Thus $\wb{Q}=C$.\\[2.3mm]
By the preceding, $(\bigoplus_{j\in J}E_j) \cap \wb{Q}=(\bigoplus_{j\in J}E_j)\cap C=Q$.
Since~$L$ is closed in $\prod_{j\in J}E_j$ and contains~$Q$, we have $\wb{Q}\sub L$.
Now $K=\prod_{j\in J}K_j$ is a closed subset of $P=\prod_{j\in J} G_j$ and
\[
\psi:=\prod_{j\in J}(\phi_j|_{K_j})\colon K\to L,\quad (x_j)_{j\in J}\mto (\phi_j(x_j))_{j\in J}\vspace{-1mm}
\]
is a homeomorphism, whence $\psi^{-1}(\wb{Q})$
is closed in~$K$  and hence also closed in $P=\prod_{j\in J}G_j$.\\[2.3mm]
Let $\wb{W}$ be the closure of $W$ in~$P$. By the preceding, $\wb{W}\sub \psi^{-1}(\wb{Q})$
and thus $\wb{W}=\psi^{-1}(\wb{Q})$, as $\psi$ is a homeomorphism.
For $x\in K$, we have $\psi(x)\in \bigoplus_{j\in J} E_j$ if and only if $x\in \bigoplus_{j\in J} G_j$.
Hence
\[
G\cap \wb{W}=G\cap\psi^{-1}(\wb{Q})=\psi^{-1}\Big(\Big(\bigoplus_{j\in J}E_j\Big)\cap \wb{Q}\Big)=
\psi^{-1}(Q)=W\sub S,
\]
entailing that also hypothesis~(iii) of Lemma~\ref{sub-complete} is satisfied.\Punkt
\section{Completeness of diffeomorphism groups}\label{secILB}
The next proposition is used to prove completeness of diffeomorphism groups.
\begin{prop}\label{henceilb}
Let $G$ be a topological group which is the projective limit
of a projective sequence $((G_n)_{n\in\N}, (q_{n,m})_{n\leq m})$
of topological Hausdorff groups.
Assume that $G_n$ admits a $C^1$-manifold structure modelled on a Banach space $(E_n,\|.\|_n)$
for each $n\in\N$, and assume that for each $n\in\N$, there exists $m\geq n$
such that the map
\[
\mu_{n,m}\colon G_m\times G_m\to G_n,\quad (x,y)\mto q_{n,m}(xy)=q_{n,m}(x)q_{n,m}(y)
\]
is~$C^1$. Then $G$ is a complete topological group.
\end{prop}
\begin{proof}
For $n\in\N$, let $q_n\colon G\to G_n$ be the limit map.
Let $(x_\alpha)_{\alpha\in A}$ be a left Cauchy net in~$G$.
By Lemma~\ref{slightgen},
it suffices to show that, for each $n\in \N$,
the left Cauchy net $(q_n(x_\alpha))_{\alpha\in A}$
converges in~$G_n$.
By hypothesis, $\mu_{n,m}$ is~$C^1$ for some $m\geq n$.
We write $\|.\|$ for the norm
\[
E_n\times E_n\to[0,\infty[,\quad (x,y)\mto\max\{\|x\|_n,\|y\|_n\}.
\]
Let $\phi\colon U\to V$ be a $C^1$-diffeomorphism
from an open identity neighbourhood $U\sub G_m$ onto an open $0$-neighbourhood $V\sub E_m$
and
$\psi\colon P\to Q$ be a $C^1$-diffeomorphism
from an open identity neighbourhood $P\sub G_n$ onto an open $0$-neighbourhood $Q\sub E_n$,
such that $\phi(e)=0$ and $\psi(e)=0$. After shrinking~$U$, we may assume that
$\mu_{n,m}(U\times U)\sub P$,
enabling us to consider the $C^1$-map
\[
f:=\psi \circ\mu_{n,m}\circ (\phi^{-1}\times \phi^{-1})\colon V\times V\to Q,\quad (x,y)\mto \psi(q_{n,m}(\phi^{-1}(x)\phi^{-1}(y))).
\]
After shrinking~$V$, we may assume that~$f$ is Lipschitz (see~\ref{C1lip}).
Let $L\geq 0$ be a Lipschitz constant for~$f$.
Then
\[
\|f(x,y)-f(x,0)\|_n\leq L\|(0,y)\|=L\|y\|_n
\]
for all $(x,y)\in V\times V$. Let
$C$ be a closed $0$-neighbourhood in~$E_n$ such that $C\sub Q$.
Then $q_n^{-1}(\psi^{-1}(C))\cap q_m^{-1}(U)$ is an identity neighbourhood in~$G$.
There is $\alpha_0\in A$ such that
\[
x_\beta^{-1}x_\alpha\in q_n^{-1}(\psi^{-1}(C))\cap q_m^{-1}(U)
\]
for all $\alpha,\beta\geq\alpha_0$. Thus
$z_\alpha:=x_{\alpha_0}^{-1} x_\alpha\in q_n^{-1}(\psi^{-1}(C))\cap q_m^{-1}(U)$
for all $\alpha\geq \alpha_0$, and $(q_n(z_\alpha))_{\alpha\geq\alpha_0}$
is a left Cauchy net in~$G_n$,
noting that
\[
z_\beta^{-1} z_\alpha=x_\beta^{-1}x_\alpha\in q_m^{-1}(U)\quad\mbox{for all $\,\alpha,\beta\geq \alpha_0$.}
\]
If we can show that $(q_n(z_\alpha))_{\alpha\geq \alpha_0}$
converges in~$G_n$, then also the subnet $(q_n(x_\alpha))_{\alpha\geq\alpha_0}$
of $(q_n(x_\alpha)_{\alpha\in A})$ will converge, as $q_n(x_\alpha)=q_n(x_{\alpha_0})q_n(z_\alpha)$.
Hence $(q_n(x_\alpha))_{\alpha\in A}$ will converge.
Since $z_\alpha=z_\beta (z_\beta^{-1}z_\alpha)$ and $z_\beta=z_\beta e$, we see that
\[
\psi(q_n(z_\alpha))-\psi(q_n(z_\beta))=
f(q_m(z_\beta),q_m(z_\beta^{-1}z_\alpha))-f(q_m(z_\beta),0)
\]
and thus
$\|\psi(q_n(z_\alpha))-\psi(q_n(z_\beta))\|_n\leq L\|q_m(z_\beta^{-1}z_\alpha)\|_n$,
which can be made arbitrarily small for large $\alpha,\beta$.
Hence $(\psi(q_n(z_\alpha)))_{\alpha\geq\alpha_0}$ is a Cauchy net in $(E_n,+)$
and thus convergent to some $w\in E_n$. Since $w_\alpha:=\psi(q_n(z_\alpha))\in C$
for all $\alpha\geq\alpha_0$ and~$C$ is closed in~$E$,
we deduce that $w\in C\sub Q$.
As a consequence, $q_n(z_\alpha)=\psi^{-1}(w_\alpha)$ converges to $\psi^{-1}(w)$.
\end{proof}
\begin{rem}\label{ilb2}
(a) In particular, Proposition~\ref{henceilb} shows that every strong (ILB)-Lie group (as in \cite{Omo})
is complete.\vspace{2mm}

(b) Since $\Diff(M)$ is a strong (ILB)-Lie group for each compact smooth manifold~$M$
without boundary (see \cite{Omo}), we deduce from (a) that $\Diff(M)$ is complete.\vspace{2mm}

(c) Let $M$ be a $\sigma$-compact finite-dimensional smooth manifold (without boundary).
If $K\sub M$ is a compact subset, then
\[
\Diff_K(M):=\{\phi\in\Diff_c(M)\colon (\forall x\in M\setminus K)\; \phi(x)=x\}
\]
is a Lie subgroup of $\Diff_c(M)$.
Morse Theory (cf.\ \cite{Hrs})
provides a compact submanifold $N\sub M$ with smooth boundary
such that $K$ is contained in the interior of~$N$.
Let $N^*$ be the double of~$N$,
which is a compact smooth manifold without boundary obtained by glueing two copies of~$N$
along their boundary. Then $\Diff_K(M)\cong \Diff_K(N^*)$ as a Lie group.\footnote{Using the double was
stimulated by discussions in~\cite{Tat}.}
As $\Diff_K(N^*)$ is a closed subgroup of the complete topological group $\Diff(N^*)$,
we see that $\Diff_K(N^*)$ (and hence also $\Diff_K(M)$) is complete.\vspace{2mm}

(d) Using Theorem A, we now obtain completeness of $\Diff_c(M)$
for $\sigma$-compact~$M$, as described in the Introduction;
applying Theorem~B to an open subgroup, completeness of $\Diff_c(M)$
for paracompact~$M$ follows.
\end{rem}
\section{Completeness of mapping groups}\label{secgauge}
We now discuss completeness of mapping groups
and test function groups.
\begin{numba}
If $k\in \N\cup\{\infty\}$ and $M$ is a $C^k$-manifold (possibly with boundary)
modelled on a Hausdorff locally convex space, we let $TM$
be the tangent bundle and recursively define $T^{j+1}M:=T(T^jM)$
for $j\in\N$ such that $j\leq k$.
If $f\colon M\to N$ is a $C^k$-map to another such manifold,
we let $Tf\colon TM\to TN$ be the tangent map
and recursively set $T^jf:=T(T^{j-1}f)\colon T^jM\to T^jN$
for all $2\leq j\in\N$ such that $j\leq k$. For convenience, $T^0M:=M$ and $T^0f:=f$.
We endow the set $C^k(M,N)$ of all $C^k$-maps from $M$ to~$N$
with the so-called \emph{compact-open $C^k$-topology}, i.e.,
the initial topology with respect to the mappings
\[
T^j\colon C^k(M,N)\to C(T^jM,T^jN),\quad f\mto T^jf
\]
for $j\in\N_0$ such that $j\leq k$, where $C(T^jM,T^jN)$ is endowed
with the compact-open topology (cf.\ \cite{NaW}).
\end{numba}
\begin{numba}
If $H$ is a Lie group,
with multiplication $\mu\colon H\times H\to H$,
then the tangent map $T\mu\colon T(H\times H)\to TH$ makes
$TH$ a Lie group, if we identify $T(H\times H)$ with $TH\times TH$ in the usual way
(see, e.g., \cite{GaN}).
Then $C^k(M,H)$ is a group for~$M$ as before and $k\in \N_0\cup\{\infty\}$,
with pointwise product
$fg:=\mu\circ (f,g)$ for $f,g\in C^k(M,H)$.
If $k\geq 1$, then
\begin{equation}\label{obssmall}
T(fg)=T\mu \circ (Tf,Tg)=Tf Tg
\end{equation}
is the product in $C(TM,TH)$, whence $Tf$ is a group homomorphism and hence
also $T^jf$ for all $j\in\N_0$ such that $j\leq k$. Since $C(T^jM,T^jH)$
is a topological group for each $j$ (see, e.g., \cite[Lemma A.5.23\,(a)]{GaN}),
we deduce that $C^k(M,H)$ is a topological group (compare also \cite{NaW}
for this discussion).
\end{numba}
Our first goal is to establish completeness properties for the topological groups
$C^k(M,H)$.
We show:
\begin{prop}\label{topolevel}
Let $H$ be a Lie group modelled on a Hausdorff locally\linebreak
convex space~$E$
and $M$ be a finite-dimensional $C^k$-manifold $($possibly with boundary$)$ for some
$k\in\N_0\cup\{\infty\}$.
If $H$ and $E$ are complete,
then also the topological group $C^k(M,H)$
is complete.
\end{prop}
The proof is based on two lemmas.
\begin{la}\label{la-mapcplte}
Given $k\in \N$, let $M$ be a finite-dimensional $C^k$-manifold
$($possibly with boundary$)$ and $N$ be a $C^k$-manifold.
Then the map
\[
\theta\colon C^k(M,N)\to C(M,N)\times C^{k-1}(TM,TN),\quad f\mto (f,Tf)
\]
is a topological embedding with closed image.
If $N$ is a Lie group, then $\theta$ is a homomorphism of groups.
\end{la}
\begin{proof}
The final observation follows from (\ref{obssmall}).\\[2.3mm]
It is clear that $\theta$ is injective. Moreover,
the topology~$\cO$ on~$C^k(M,N)$ making $\theta$ a topological embedding is initial
with respect to the inclusion map
\[
T^0\colon C^k(M,N)\to C(M,N)
\]
and $T\colon C^k(M,N)\to C^{k-1}(TM,TN)$.
As the topology on $C^{k-1}(TM,TN)$ is initial with respect to
the maps $T^j\colon C^{k-1}(TM,TN)\to C(T^{j+1}M,T^{j+1}N)$
for $j\in\{0,\ldots, k-1\}$,
we deduce with the well-known transitivity of initial topologies (see, e.g., \cite[Lemma~A.2.7]{GaN})
that $\cO$ is initial with respect to $T^0$ and the maps $T^j\circ T\colon C^k(M,N)\to C(T^{j+1}M,T^{j+1}N)$
for $j\in\{0,\ldots, k-1\}$. Hence~$\cO$ coincides with the compact-open $C^k$-topology.\\[2.3mm]
To see that $\theta$ has closed image, let $(f_\alpha,Tf_\alpha)_{\alpha\in A}$ be a net
in the image of~$\theta$ which converges to some $(f,g)\in C(M,N)\times C^{k-1}(TM,TN)$.
It now suffices to show that each $x\in M$ has an open neighbourhood $U$ such that
$f|_U$ is $C^1$ and $T(f|_U)=g|_{TU}$; then $f$ is $C^k$ and $g=Tf$,
whence $(f,g)=\theta(f)$ is in the image of~$\theta$.\\[2.3mm]
For $x\in M$, there is a chart $\psi\colon U_\psi\to V_\psi\sub Y$ of~$N$ such that $f(x)\in U_\psi$
(where $Y$ is the
modelling space of~$N$) and a chart $\phi\colon U_\phi\to V_\phi\sub X$ of~$M$
with $x\in U_\phi$ such that $U_\phi$ has compact closure $K:=\wb{U_\phi}$
and $f(\wb{U_\phi})\sub U_\psi$ (where~$X$ is the modelling space of~$M$).
As the compact-open $C^k$-topology on $C^k(M,N)$ is finer than
the compact-open topology, the set
\[
W:=\{h\in C^k(M,N)\colon h(K)\sub U_\psi\}
\]
is an open neighbourhood of~$f$ in $C^k(M,N)$.
Thus, we find $\alpha_0\in A$ such that $f_\alpha\in W$
for all $\alpha\in A$ such that $\alpha\geq\alpha_0$.
For such~$\alpha$, we can define
\[
h_\alpha:=\psi\circ f_\alpha \circ\phi^{-1}\colon V_\phi\to Y.
\]
Then $h_\alpha\to \psi\circ f\circ \phi^{-1}$ and $d(h_\alpha)\to d\psi\circ g|_{TU_\phi}\circ T\phi^{-1}$
uniformly on compact sets (see Lemmas~A.5.3,
A.5.5 and A.5.9 in~\cite{GaN}),
entailing that $h:=\psi\circ f\circ \phi^{-1}$ is $C^1$ with
$dh=d\psi\circ g|_{TU_ \phi}\circ T\phi^{-1}$
(by \cite[Lemma~1.4.16]{GaN}),\footnote{To apply the
lemma, give $\{\alpha\in A\colon \alpha\geq \alpha_0\} \cup\{\infty\}$ the topology with $\{\alpha\}$
for $\alpha\in A$ with $\alpha\geq\alpha_0$
and $\{\alpha\in A\colon \alpha\geq\beta\}\cup\{\infty\}$ for $\beta\in A$ with $\beta\geq\alpha_0$
as a basis (which is a Hausdorff topology).
Lemma 1.4.16 and Proposition A.5.17 from \cite{GaN} imply the assertion.}
and thus $Th=T\psi\circ g|_{TU_\phi}\circ T\phi^{-1}$.
Hence $f|_{U_\phi}$ is $C^1$ with $T(f|_{U_\phi})=g|_{TU_\phi}$.
\end{proof}
\begin{la}\label{la2map}
Let $M$ and $N$ be smooth manifolds $($possibly with boundary$)$, both modelled on Hausdorff
locally convex spaces. Then
\[
C^\infty(M,N)\, =\,\pl_{k\in\N_0} \!C^k(M,N)\vspace{-.8mm}
\]
as a topological space, using the respective inclusion maps as
the bonding maps and limit maps.
\end{la}
\begin{proof}
Consider the standard realization $P\sub \prod_{k\in\N_0}C^k(M,N)$
of the projective limit. As all bonding maps are the inclusion maps, it is the set
of all sequences $(f_k)_{k\in\N_0}\in\prod_{k\in\N_0}C^k(M,N)$ such that
\[
f_j=f_k\quad\mbox{for all $j,k\in\N_0$ such that $j\leq k$.}
\]
Then $f_0=f_k$ for all $k\in\N_0$ and thus $f_0\in C^\infty(M,N)$,
entailing that the map
\[
\Phi\colon C^\infty(M,N)\to P,\quad f\mto (f)_{k\in\N_0}
\]
is a bijection. The topology $\cO$ on $C^\infty(M,N)$ making $\Phi$ a homeomorphism
is initial with respect the compositions $T^j\circ\pi_k\circ\Phi=T^j$ for $k\in\N_0$ and
$j\in\N_0$ such that $j\leq k$ (where $\pi_k$ is the projection from the direct
product onto its $k$th factor). It therefore coincides with the compact-open $C^\infty$-topology.
\end{proof}
{\bf Proof of Proposition~\ref{topolevel}.}
If we can show that $C^k(M,H)$ is complete for each $k\in\N_0$, then also
$C^\infty(M,H)$ (which is the projective limit of the latter topological groups, by Lemma~\ref{la2map})
will be complete. We proceed by induction. If $k=0$, then $C^0(M,H)=C(M,H)$
is complete since~$H$ is complete and~$M$, being locally compact,
is a $k_\R$-space\footnote{Recall that a topological space~$X$ is called a \emph{$k_\R$-space}
if it is Hausdorff and functions $f\colon X\to\R$
are continuous if and only if $f|_K$ is continuous for each compact subset $K\sub X$.}
(see, e.g., \cite[Lemma~A.5.23\,(d)]{GaN}).\\[2.3mm]
If $k\in\N$ and the assertion holds for $k-1$ in place of~$k$,
then $C^{k-1}(TM,TH)$ is complete as $TM$ has finite dimension,
$TH\cong L(H)\rtimes H$ is complete
(see \ref{factscompl}\,(e))
and also its modelling space $E\times E$ is complete.
Moreover, $C(M,H)$ is complete. As,
by Lemma~\ref{la-mapcplte},
the topological group $C^k(M,H)$
is isomorphic to a closed subgroup of the direct product $C(M,H)\times C^{k-1}(M,TH)$
of complete groups, also $C^k(M,H)$ is complete.\Punkt
\begin{rem}\label{thereafter}
(a) If $M$ is a compact smooth manifold and~$H$ a Lie group, then the topology on the Lie group
$C^k(M,H)$ (for $k\in\N_0\cup\{\infty\}$) coincides with the compact-open $C^k$-topology
defined above (see \cite{GaN}). Hence $C^k(M,H)$ is complete whenever $H$ and its modelling space
are complete.\vspace{2mm}

(b) If $M$ is a finite-dimensional $\sigma$-compact
smooth manifold and $K\sub M$ a compact subset, then the topology on the Lie group
\[
C^k_K(M,H):=\{\gamma\in C^k(M,H)\colon \Supp(\gamma)\sub K\}
\]
is induced by the compact-open $C^k$-topology on $C^k(M,H)$.
Since $C^k_K(M,H)$ is a closed subgroup of $C^k(M,H)$,
we deduce that $C^k_K(M,H)$ is complete whenever $H$
and its modelling space are complete.
\end{rem}
\begin{prop}\label{mapscomplete}
Let $M$ be a paracompact finite-dimensional smooth manifold
and $H$ be a Lie group. If $H$ and its modelling space are complete,
then $C^k_c(M,H)$ is complete for each $k\in\N_0\cup\{\infty\}$.
\end{prop}
\begin{proof}
If $M$ is $\sigma$-compact, we choose a sequence $(K_n)_{n\in\N}$ of compact subsets
of~$M$ such that $M=\bigcup_{n\in\N} K_n$ and $K_n\sub K_{n+1}^0$ for each~$n$.
Then
\[
C^k_{K_1}(M,H)\sub C^k_{K_2}(M,H)\sub \cdots
\]
is a strict direct sequence of topological groups and product sets
are large in $C^k_c(M,H)=\bigcup_{n\in\N}C^k_{K_n}(M,H)$
as the product map
\[
\bigoplus_{n\in\N}C^k_{K_n}(M,H)\to C^k_c(M,H),\;\;(\gamma_1,\ldots,\gamma_n,e,e,\ldots)\mto
\gamma_1\gamma_2\cdots\gamma_n
\]
admits a smooth local section around~$e$ which takes $e$ to $e$
(see \cite[Example 11.6 and Remark 11.5]{JFA}).
Since $C^k_{K_n}(M,H)$ is complete for each $n\in\N$ (see Remark~\ref{thereafter}\,(b)),
we deduce with Theorem~A that $C^k_c(M,H)$ is complete.\\[2.3mm]
If $M$ is merely paracompact, we let $(M_j)_{j\in J}$ be the family of connected components
of~$M$ (each of which is $\sigma$-compact). Then the map
\[
\Phi\colon C^k_c(M,H)\to\bigoplus_{j\in J} C^k_c(M_j,H),\quad \gamma\mto (\gamma|_{M_j})_{j\in J}
\]
is an isomorphism of groups and we give $C^k_c(M,H)$ the smooth Lie group structure which turns
$\Phi$ into an isomorphism of Lie groups. As the weak direct product is complete by
the first part of the proof and Theorem~B, we see that also $C^k_c(M,H)$ is complete.
\end{proof}
\begin{rem}\label{alsogauge}
Let $H$ be a Lie group, $M$ be a smooth manifold of dimension $m\in\N$ and
$P\to M$ be a smooth principal
bundle with structure group~$H$.\vspace{2.3mm}

(a) If $M$ is $\sigma$-compact and the condition $\mbox{SUB}_\oplus$
of~\cite{Sch} is satisfied,\footnote{This is automatic if~$H$
is \emph{locally exponential} in the sense that $H$ has a smooth exponential function
which is a local $C^\infty$-diffeomorphism at~$0$.}
then the gauge group $\Gau_c(P)$ of~$P$ is a Lie group
which is isomorphic to a closed Lie subgroup of the weak direct product
$\bigoplus_{n\in\N}C^\infty(K_n,H)$,
where $(K_n)_{n\in\N}$ is a locally finite cover
of~$M$ by $m$-dimensional compact smooth submanifolds $K_n$ with boundary
such that $P$ is trivializable on some open neighbourhood of~$K_n$.
If $H$ and its modelling space are complete, then
$C^\infty(K_n,H)$ is complete for each $n\in\N$
(by Proposition~\ref{topolevel}),
whence also the weak direct product is complete (by Theorem~B)
and hence also $\Gau_c(P)$, being isomorphic to a closed subgroup of the latter
as a topological group.
Then the full group $\Aut_c(P)$ of compactly supported symmetries of~$P$
(which was made a Lie group in~\cite{Sch})\footnote{For compact~$M$, the Lie group
$\Aut(P)$ was already constructed in~\cite{Woc}.}
is complete,
as it is an extension
\[
\{e\}\to\Gau_c(P)\to \Aut_c(P)\to \Diff_c(M)_P\to\{e\}
\]
of Lie groups (and hence of topological groups)
for some open subgroup $\Diff(M)_P\sub \Diff_c(M)$.
Since $\Diff_c(M)$ is complete (as already observed) and also $\Gau_c(P)$ is complete, so is the extension
$\Aut_c(P)$ (as recalled in \ref{factscompl}\,(d)).\vspace{2.3mm}

(b) If $M$ is paracompact and condition $\mbox{SUB}_\oplus$
is satisfied by $P|_C$ for each connected component $C$ of~$M$,
let $(M_j)_{j\in J}$ be the family of connected components of~$M$.
We can identify $\Gau_c(P)$ with the weak direct product $\bigoplus_{j\in J}\Gau_c(M_j)$
(whence it can be considered as a complete Lie group by (a) and Theorem~B).
Moreover, $\Aut_c(P)$ can be made a Lie group having the weak direct product
$\bigoplus_{j\in J}\Aut_c(P|_{M_j})$ as an open subgroup. Hence $\Aut_c(P)$ is complete,
using Theorem~B.
\end{rem}
\section{\!\!Product sets in unions of Banach-Lie groups}\label{unionBan}
We now discuss ascending unions of Banach-Lie groups.
\begin{numba}\label{presett}
Let $G_1\sub G_2\sub\cdots$ be analytic Banach-Lie groups over $\K\in\{\R,\C\}$
such that the inclusion maps $j_{n+1,n}\colon G_n\to G_{n+1}$ are $\K$-analytic group
homomorphisms. Identifying the Banach-Lie algebra $\cg_n:=L(G_n)$ with the
image of the map $L(j_{n+1,n})$ in~$\cg_{n+1}$, we can consider the ascending union
$\cg:=\bigcup_{n\in\N}\cg_n$
and endow it with the locally convex direct limit topology. Give $G:=\bigcup_{n\in\N}G_n$
the unique group structure making each inclusion map $G_n\to G$
a group homomorphism.
Define
$\exp_G\colon \cg\to G$
piecewise via $\exp_G(x):=\exp_{G_n}(x)$ if $x\in\cg_n$.
\end{numba}
\begin{numba}\label{Dahsett}
(Dahmen's setting). \emph{If, in the situation of} \ref{presett},
\begin{itemize}
\item[\rm(a)]
\emph{$\cg$ is Hausdorff};
\item[\rm(b)]
\emph{There are norms $\|.\|_n$ on $\cg_n$ defining its topology for $n\in\N$,
such that the Lie bracket of $\cg_n$ and each inclusion map
$(\cg_n,\|.\|_n)\to(\cg_{n+1},\|.\|_{n+1})$ has operator norm $\leq 1$; and}
\item[\rm(c)] \emph{$\exp_G$ is injective on some $0$-neighbourhood,}
\end{itemize}
\emph{then $G$ admits a unique $\K$-analytic Lie group
structure such that $P:=\exp_G(Q)$ is open in~$G$ for some open $0$-neighbourhood
$Q\sub\cg$ and $\exp_G|_Q^P$ a diffeomorphism of $\K$-analytic manifolds.} (See \cite[Theorem~C]{Da1}).
\end{numba}
\begin{prop}\label{arelarge}
Let $G_1\sub G_2\sub\cdots$ be analytic Banach-Lie groups over $\K\in\{\R,\C\}$
such that the inclusion maps $G_n\to G_{n+1}$ are $\K$-analytic group
homomorphisms. Assume that Dahmen's conditions \emph{(a)--(c)} from~\emph{\ref{Dahsett}}
are satisfied and
endow $G$ with the $\K$-analytic Lie group structure described there.
Let $\cO$ be the topology on the Lie group~$G$.
Then product sets are large in $(G,\cO)=\bigcup_{n\in\N}G_n$.
As a consequence, $\cO=\cO_{TG}$ holds, i.e., $\cO$ makes~$G$ the direct limit topological
group $\dl G_n$.\vspace{-.7mm} If, moreover, the direct sequence
$G_1\sub G_2\sub\cdots$ is strict, then $(G,\cO)$ is complete.
\end{prop}
Before we prove Proposition~\ref{arelarge},
let us compile useful facts concerning the Baker-Campbell-Hausdorff (BCH-)
multiplication.
\begin{numba} (See \cite{Bou}).
Let $\cg$ be a Banach-Lie algebra and $\|.\|$ be a norm on~$\cg$ which is \emph{compatible}
in the sense that it defines the topology of~$\cg$ and $\|[x,y]\|\leq\|x\|\,\|y\|$ holds
for all $x,y\in \cg$.
Then the BCH-series converges for $x,y\in\cg$ with $\|x\|+\|y\|<\ln\frac{3}{2}$
and defines an analytic function
\[
\{(x,y)\in\cg\times\cg\colon \|x\|+\|y\|<\ln(3/2)\}\to B^\cg_{\ln 2}(0),\quad (x,y)\mto x*y.
\]
If $\cg=L(G)$ for some Banach-Lie group~$G$, then
\begin{equation}\label{expBCH}
\exp_G(x*y)=\exp_G(x)\exp_G(y)\;\mbox{for all $x,y\in\cg$ with $\|x\|+\|y\|<\ln\frac{3}{2}$.}
\end{equation}
\end{numba}
See \cite[Lemma~3.5\,(a)]{Dah} for the following estimates concerning
derivatives of the BCH-multiplication:
\begin{la}\label{BCH}
There exists $s_0\in \;]0,\frac{1}{3}\ln\frac{3}{2}[$
such that,
for each Banach-Lie\linebreak
algebra $\cg$ and compatible norm $\|.\|$ on~$\cg$,
\begin{equation}\label{derBC}
(\forall x,y\in B^\cg_{s_0}(0))\;\; \|(\mu^\cg)'(x,y)-\alpha^\cg \|_{op}\leq\frac{1}{2},
\end{equation}
where
$\alpha^\cg\colon\cg\times\cg\to\cg$, $(x,y)\mto x+y$
is addition
and
$\mu^\cg\colon B^\cg_{s_0}(0)\times B^\cg_{s_0}(0)\to\cg$, $(x,y)\mto x*y$
the BCH-multiplication.\Punkt
\end{la}
To calculate the operator norm, the maximum norm was used on
$\cg\times\cg$ here.\\[2.3mm]
With $s_0$ and notation as in Lemma~\ref{BCH},
we deduce:
\begin{la}\label{BCHballs}
For each Banach-Lie algebra~$\cg$ and compatible norm $\|.\|$ on $\cg$,
we have
\begin{equation}\label{doesjob}
x*y+B^\cg_{r/2}(0)\sub x*B_r^\cg(y)\sub x*y+B^\cg_{3r/2}(0)
\end{equation}
for all $x\in B^\cg_{s_0}(0)$,
$y\in B^\cg_{s_0/2}(0)$ and $r\in\,]0,\frac{s_0}{2}]$.
\end{la}
\begin{proof}
Setting $R(x,y):=\mu^\cg(x,y)-x-y$, we have
\[
\mu^\cg(x,y)=x+y+R(x,y)\quad\mbox{for all $x,y\in B^\cg_{s_0}(0)$.}
\]
Since $\| R'(x,y)\|_{op}\leq\frac{1}{2}$ for all $(x,y)\in B^\cg_{s_0}(0)\times B^\cg_{s_0}(0)$
by (\ref{derBC})
and the latter set is convex, \ref{C1lip}\,(b)
shows that
$\Lip(R)\leq\frac{1}{2}$.
For $x\in B^\cg_{s_0}(0)$, consider the map
$\mu^\cg_x\colon B^\cg_{s_0}(0)\to\cg$, $y\mto \mu^\cg(x,y)$.
For all $y,z\in B^\cg_{s_0}(0)$, we have
\begin{eqnarray*}
\|\mu^\cg_x(z)-\mu^\cg_x(y)-\id_\cg(z-y)\|
&=&\|\mu^\cg(x,z)-(x+z)-\mu^\cg(x,y)+x+y\|\\
&=&\|R(x,z)-R(x,y)\|\leq\Lip(R)\|z-y\|
\end{eqnarray*}
and thus $\Lip(\mu^\cg_x-\id_\cg)\leq\Lip(R)\leq\frac{1}{2}$.
Applying now the Quantitative Inverse Function Theorem
\cite[Lemma~6.1\,(a)]{IMP} (or the version in \cite{Wel})
to the function $\mu^\cg_x$ with $A:=\id_\cg$, we get~(\ref{doesjob}).
\end{proof}
{\bf Proof of Proposition~\ref{arelarge}.}
To see that product sets are large in $G=\bigcup_{n\in\N} G_n$,
let $(U_n)_{n\in\N}$ be a sequence of identity neighbourhoods $U_n\sub G_n$.
By hypothesis,
\begin{equation}\label{hypcontr}
\|x\|_m\leq \|x\|_k\quad\mbox{for all integers $1\leq k\leq m$ and all $x\in\cg_k$.}
\end{equation}
Let $s_0$ be as in Lemma~\ref{BCH}.
For $n\in\N$, choose
\begin{equation}\label{ther}
r_n\in \,]0,s_0/2^{n+1}[
\end{equation}
so small that
\begin{equation}\label{inside}
V_n:=\exp_{G_n}(B^{\cg_n}_{r_n}(0))\sub U_n.
\end{equation}
Write $x *_n y:=\mu^{\cg_n}(x,y)$ for the BCH-multiplication, for $x,y\in B^{\cg_n}_{s_0}(0)$
(as in Lemma~\ref{BCH}).
Define
$W_1:=B^{\cg_1}_{r_1}(0)$.
We claim that
\begin{equation}\label{wantdefine}
W_n:=W_{n-1} *_n B^{\cg_n}_{r_n}(0)
\end{equation}
can be defined for each integer $n\geq 2$, and
\begin{equation}\label{enabdfn}
\sum_{k=1}^n B^{\cg_k}_{r_k/2}(0)\sub W_n\sub \sum_{k=1}^n B^{\cg_k}_{3r_k/2}(0).
\end{equation}
If the claim is true, then $W:=\bigcup_{n\in\N} W_n$ is a $0$-neighbourhood in $\cg$, as
it contains the convex set
\[
S:=\bigcup_{n\in\N} (B^{\cg_1}_{r_1/2}(0)+\cdots+ B^{\cg_n}_{r_n/2}(0))
\]
which is an open $0$-neighbourhood in the locally convex direct limit
$\cg=\bigcup_{n\in\N}\cg_n$ as it intersects each $\cg_n$ in an open $0$-neighbourhood.
Since $\exp_G(W)$ contains the open subset $\exp_G(S\cap Q)$ of~$G$
(with $Q$ as in \ref{Dahsett}), we deduce that $\exp_G(W)$ is an identity
neighbourhood in~$G$. Now
\[
\exp_G(W_n)=V_1V_2\cdots V_n
\]
for each $n\in \N$; this is trivial if $n=1$ and follows inductively
as
\begin{eqnarray*}
\exp_G(W_n) &=&\exp_{G_n}(W_n)=\exp_{G_n}(W_{n-1} *_n B^{\cg_n}_{r_n}(0))\\
&=&\exp_{G_n}(W_{n-1})\exp_{G_n}(B^{\cg_n}_{r_n}(0))
=\exp_{G_{n-1}}(W_{n-1})V_n\\
&=&V_1\cdots V_{n-1}V_n,
\end{eqnarray*}
using (\ref{expBCH}), the definition of~$V_n$, and the inductive hypothesis.
Thus
\[
U:=\bigcup_{n\in\N} U_1\cdots U_n\supseteq\bigcup_{n\in\N}V_1\cdots V_n=
\bigcup_{n\in\N}\exp_G(W_n)=\exp_G(W),\vspace{-1mm}
\]
whence $U$
is an identity neighbourhood in~$G$ and so product sets are large.\\[2.3mm]
We now prove the claim, by induction.
For $n=2$, we can form
$W_2:=W_1*_2 B^{\cg_2}_{r_2}(0)$
as $W_1=B^{\cg_1}_{r_1}(0)\sub B^{\cg_1}_{s_0}(0)\sub B^{\cg_2}_{s_0}(0)$
by (\ref{hypcontr}), and $B^{\cg_2}_{r_2}(0)\sub B^{\cg_2}_{s_0}(0)$.
Moreover, as $r_2\leq s_0/2$, we have
$W_1+B^{\cg_2}_{r_2/2}(0)\sub W_2\sub W_1+ B^{\cg_2}_{3r_2/2}(0)$
by (\ref{doesjob}). Hence
\[
B^{\cg_1}_{r_1/2}(0)+B^{\cg_2}_{r_2/2}(0)\sub W_2\sub B^{\cg_1}_{3r_1/2}(0)
+ B^{\cg_2}_{3r_2/2}(0)
\]
a fortiori. For the induction step, assume that $n\geq 2$ and that $W_1,\ldots,W_n$ have already been defined such that (\ref{enabdfn})
holds with $k\in\{1,\ldots,n\}$ in place of~$n$.
In particular, (\ref{enabdfn}) holds for $n$
and its right hand side is a subset of
\[
\sum_{k=1}^n B^{\cg_k}_{s_0/2^k}(0)\sub
\sum_{k=1}^n B^{\cg_{n+1}}_{s_0/2^k}(0)\sub B^{\cg_{n+1}}_{s_0}(0),
\]
using (\ref{hypcontr}) for the first inclusion.
Thus $W_n\sub B^{\cg_{n+1}}_{s_0}(0)$ and
since $r_{n+1}\leq s_0$, we deduce that $W_{n+1}:=W_n *_{n+1} B^{\cg_{n+1}}_{r_{n+1}}(0)$
can be defined. Moreover,
\[
W_n+ B^{\cg_{n+1}}_{r_{n+1}/2}(0)\sub W_{n+1}\sub W_n+ B^{\cg_{n+1}}_{3r_{n+1}/2}(0),
\]
by (\ref{doesjob}). Using (\ref{enabdfn}), we obtain
$\sum_{k=1}^{n+1} B^{\cg_k}_{r_k/2}(0)\sub W_{n+1}\sub \sum_{k=1}^{n+1} B^{\cg_k}_{3r_k/2}(0)$,
which completes the inductive proof of the claim.\\[2.3mm]
As product sets are large in $G=\bigcup_{n\in\N}G_n$ by the preceding,
the last and penultimate assertion of the proposition follow
from \cite[Proposition~11.8]{JFA} and Theorem~A, respectively.\Punkt
{\small
{\bf Helge  Gl\"{o}ckner}, Institut f\"{u}r Mathematik, Universit\"at Paderborn,\\
Warburger Str.\ 100, 33098 Paderborn, Germany; {\tt  glockner@math.upb.de}}\vfill
\end{document}